\documentclass[12pt]{amsart} 
\usepackage{amsmath,amsthm,amssymb,mathrsfs}
 
\numberwithin{equation}{section}

\providecommand{\binom}[2]{{#1\choose#2}}

\renewcommand{\geq}{\geqslant}
\renewcommand{\leq}{\leqslant}
\newcommand{\Osh}{{\mathcal O}}                        

\renewcommand{\H}{\mathrm{H}}                          

\newcommand{\id}{\operatorname{id}}                    

\newcommand{\K}{\mathrm{K}}                            

\newcommand{\R}{\operatorname{R}}
\newcommand{\Ish}{\mathcal{I}}

\newcommand{\NS}{\operatorname{NS}} 
\newcommand{\End}{\operatorname{End}} 
\newcommand{\kk}{\mathbf{k}}

\newcommand{\Nrd}{\operatorname{Nrd}}
\newcommand{\N}{\operatorname{N}}
\renewcommand{\emptyset}{\varnothing}
\newcommand{\PP}{\mathbb{P}} 
\newcommand{\QQ}{\mathbb{Q}} 
\newcommand{\RR}{\mathbb{R}} 
\newcommand{\ZZ}{\mathbb{Z}} 

\newtheorem{theorem}{Theorem}[section]
\newtheorem{lemma}[theorem]{Lemma}
\newtheorem{corollary}[theorem]{Corollary}
\newtheorem{proposition}[theorem]{Proposition}
\theoremstyle{definition}
\newtheorem{defn}[theorem]{Definition}
\newtheorem{remark}[theorem]{Remark}
\newtheorem{example}[theorem]{Example}

\begin{document}

\title[Wedderburn components and index theory]{Wedderburn components, the index theorem and continuous Castelnuovo-Mumford regularity for semihomogeneous vector bundles
}
\author{Nathan Grieve}
\address{Department of Mathematics \& Computer Science,
Royal Military College of Canada, P.O. Box 17000,
Station Forces, Kingston, ON, K7K 7B4, Canada
}
\address{
School of Mathematics and Statistics, 4302 Herzberg Laboratories, Carleton University, 1125 Colonel By Drive, Ottawa, ON, K1S 5B6, Canada
}
\address{D\'{e}partement de math\'{e}matiques, Universit\'{e} du Qu\'{e}bec \'a Montr\'{e}al, Local PK-5151, 201 Avenue du Pr\'{e}sident-Kennedy, Montr\'{e}al, QC, H2X 3Y7, Canada}
\email{nathan.m.grieve@gmail.com}%

\begin{abstract} 
We study the property of \emph{continuous Castelnuovo-Mumford regularity}, for semihomogeneous vector bundles over a given Abelian variety, which was formulated in A. K\"{u}ronya and Y. Mustopa [Adv. Geom. \textbf{20} (2020), no.~3,
  401--412].  Our main result gives a novel description thereof.  It is expressed in terms of certain normalized polynomial functions that are obtained via the Wedderburn decomposition of the Abelian variety's endomorphism algebra.  This result builds on earlier work of Mumford and Kempf and applies the form of the Riemann-Roch Theorem that was established in N. Grieve [New York J. Math. \textbf{23} (2017), 1087--1110].  In a complementary direction, we explain how these topics pertain to the \emph{Index} and \emph{Generic Vanishing Theory} conditions for simple semihomogeneous vector bundles.  In doing so, we refine results from M.  Gulbrandsen [Matematiche (Catania) \textbf{63}
  (2008), no.~1, 123--137], N. Grieve [Internat. J. Math. \textbf{25} (2014), no.~4, 1450036, 31] and D. Mumford [Questions on
  {A}lgebraic {V}arieties ({C}.{I}.{M}.{E}., {III} {C}iclo, {V}arenna, 1969),
  Edizioni Cremonese, Rome, 1970, pp.~29--100].  
\end{abstract}
\thanks{
I thank the Natural Sciences and Engineering Research Council of Canada for their support through my grants DGECR-2021-00218 and RGPIN-2021-03821. 
}
\maketitle

\section{Introduction}

Recall, that a coherent sheaf $\mathcal{F}$ on a projective variety $X$ is \emph{$m$-regular} with respect to a globally generated ample line bundle $\Osh_X(1)$, if 
$$
\H^i(X,\mathcal{F}(m-i)) = 0 \text{, for all $i > 0$.}
$$
The concept of $m$-regularity was formulated by Mumford  \cite[Lecture 14]{Mum66}.  It remains a fundamental cohomological invariant.  For example, if $\mathcal{F}$ is $m$-regular and  $k > m$, then $\H^0(X,\mathcal{F}(k))$ is spanned by the image of the natural map
$$
\H^0(X,\mathcal{F}(k-1)) \otimes \H^0(X,\Osh_X(1)) \rightarrow \H^0(X,\mathcal{F}(k)).
$$  
In particular, if $\mathcal{F}$ is $m$-regular and $k \geq m$, then 
$$
\mathcal{F}(k) := \mathcal{F} \otimes_{\Osh_X} \Osh_X(k)
$$
is a globally generated $\Osh_X$-module.

The \emph{Castelnuovo-Mumford regularity} of $\mathcal{F}$, with respect to $\Osh_X(1)$, is 
defined to be the smallest integer $m$ for which $\mathcal{F}$ is $m$-regular.  Effective criteria for $m$-regularity is a foundational problem within algebraic geometry.   A starting point is the following theorem of Mumford \cite[p.~101]{Mum66}.

\begin{theorem}[{\cite[p.~101]{Mum66}}]
For all nonnegative integers $n$, there exists a polynomial $F_n(x_0,\dots,x_n)$ so that for all coherent sheaves of ideals $\Ish$ on $\PP^n$, if $a_0,\dots,a_n \in \ZZ$ are defined by the condition that
$$
\chi(\Ish(m)) = \sum_{i=0}^n(-1)h^i(\PP^n,\Ish(m)) = \sum_{i=0}^n a_i \binom{m}{i} \text{, }
$$ 
then $\Ish$ is $F_n(a_0,\dots,a_n)$-regular with respect to the tautological line bundle $\Osh_{\PP^n}(1)$.
\end{theorem}

Computationally effective methods for calculation of the Castelnuovo-Mumford regularity, $\operatorname{reg}_{\Osh_{\PP^n}(1)}(\mathcal{F})$, for coherent sheaves $\mathcal{F}$ on projective $n$-space $\PP^n$, follow from work of Bayer and Stillman \cite[Theorem 1.10]{Bayer:Stillman:1987}.  We refer to \cite[Section 1.8]{Laz1} for more details about Castelnuovo-Mumford regularity.

Turning to the context of Abelian varieties, and more generally irregular varieties, it was noted by Green and Lazarsfeld, in \cite{Green:Lazarsfeld:87:GV}, building on work of Mumford, \cite[Lecture 14]{Mum66}, Mukai, \cite{Muk78} and \cite{Mukai-duality}, among others, that measures of cohomological positivity and criterion for global generation of coherent sheaves, with respect to  ample line bundles, is achieved via the \emph{Generic Vanishing Theory}.

A systematic development of that viewpoint, from the perspective of the Fourier-Mukai transform was initiated by Hacon \cite{Hacon:2004}.  It was developed further in a series of articles by Pareschi and Popa (including \cite{Pareschi:Popa:2003}, \cite{Par:Popa:II} and \cite{Par:Popa:III}).  As one example, the property of \emph{Mukai regularity}, see Section \ref{GV:theory:prelim}, for sheaves on a given Abelian variety, was introduced in \cite{Pareschi:Popa:2003}.  A main result is the \emph{$M$-regularity criterion} \cite[p. 285]{Pareschi:Popa:2003}.

\begin{theorem}[{\cite[p.~285]{Pareschi:Popa:2003}}]
Let $\mathcal{F}$ be a coherent sheaf and $\mathcal{L}$ a line bundle on an Abelian variety $A$.  If $\mathcal{F}$ and $\mathcal{L}$ are $M$-regular $\Osh_A$-modules, then $\mathcal{F} \otimes \mathcal{L}$ is globally generated.
\end{theorem}

More recently, motivated by work of Barja, Pardini and Stoppino, \cite{Barja:Pardini:Stoppiono:2020}, K\"{u}ronya and Mustopa, in \cite{Kuronya:Mustopa:2020}, formulated a concept of \emph{continuous Castelnuovo-Mumford regularity,} denoted by $\operatorname{reg}_{\mathrm{cont}}(\mathcal{F},\Osh_X(1))$, for coherent sheaves $\mathcal{F}$ on a given polarized irregular variety $(X, \Osh_X(1))$.  Briefly, this is defined to be the smallest integer $m$ for which the \emph{cohomological support loci} 
$ \operatorname{V}^i(X,\mathcal{F}(m-i))\text{,}$ for all $i>0$, are proper Zariski closed subsets of $\operatorname{Pic}^0(X)$, 
the identity component of the Picard group.  

For the case of Abelian varieties,  K\"{u}ronya and Mustopa's main result, \cite[Theorem A]{Kuronya:Mustopa:2020}, implies that  continuous Castelnuovo-Mumford regularity is a numerical property for semihomogeneous vector bundles.  It builds on \cite{Grieve-cup-prod-ab-var}.   

\begin{theorem}[{\cite[See Theorem A and Theorem B]{Kuronya:Mustopa:2020}}]\label{Kuronya:Mustopa:ThmA}
Let $(A,\Osh(1))$ be a polarized Abelian variety.  The following assertions hold true.
\begin{enumerate}
\item[(i)]{Then there exists a piecewise constant function
\begin{equation}\label{cont:reg:function}
\rho_{\Osh(1)} \colon \N^1_\RR(A) \rightarrow \ZZ
\end{equation}
which has the property that
$$
\operatorname{reg}_{\mathrm{cont}}(\mathcal{E},\Osh(1)) = \rho_{\Osh(1)}(\operatorname{det}(\mathcal{E}) / \operatorname{rank}(\mathcal{E}))
$$
for each semihomogeneous vector bundle $\mathcal{E}$ over $A$.  (Here, we identify $\Osh(1)$ and $\operatorname{det}(\mathcal{E})$ with their classes in $\N^1_{\RR}(A)$ the real N\'eron-Severi space of $A$.)
}
\item[(ii)]{If $\mathcal{E}$ is a semihomogeneous vector bundle on $A$ with the property that the class of $\det(\mathcal{E})$ is a rational multiple of the class of $\Osh(1)$, then $\operatorname{reg}_{\operatorname{cont}}(\mathcal{E},\Osh(1))$ is equal to the smallest integer $m$ for which $\mathcal{E}(m-g)$ is a generic vanishing sheaf.
}
\end{enumerate}
\end{theorem}

Our purpose here is to build on, and refine, these results from \cite{Kuronya:Mustopa:2020}. For instance, note that our formulation of Theorem \ref{Kuronya:Mustopa:ThmA}, does not require global generation on the polarizing line bundle.  Nor does it ask that the algebraically closed base field be of characteristic zero.  Moreover, 
Theorem \ref{reduced:norm:cont:reg}, below, makes explicit the manner in which the function \eqref{cont:reg:function} depends on both the Wedderburn decomposition and the isogney class of the given Abelian variety.  A key point is \cite[Corollary 4.2]{Grieve-cup-prod-ab-var}, which builds upon the index theorem of Mumford \cite[Chapter 16]{Mum:v1} and \cite[Appendix]{Mum:Quad:Eqns}.  

Indeed, in the present article, we apply the main result of \cite{Grieve-cup-prod-ab-var} to show how Albert's Theorem and the Poincar\'{e} Reducibility Theorem, for a given Abelian variety, are reflected in these cohomological properties for semihomogeneous vector bundles.
Our results here allow for an explicit determination of the property of continuous Castelnuovo-Mumford regularity for semihomogeneous vector bundles.  It complements the numerical description from  \cite{Kuronya:Mustopa:2020}.  It is expressed in terms of certain normalized polynomials that are determined by the reduced norms of the Wedderburn components of the endomorphism algebra.

As some additional results, which are of an independent interest, we build on the works \cite{Grieve-cup-prod-ab-var}, \cite{Grieve:theta} and \cite{Grieve:R-R:abVars}, for example, which have origins in Mumford's index theorem for line bundles on Abelian varieties \cite[p.~ 156]{Mum:v1}.  In this regard, our main results are Theorem \ref{semi:homog:index:thm} and Corollary \ref{Semi:homog:RR}.  Together, they improve upon  \cite[Appendix Theorem 2]{Mum:Quad:Eqns} and the main results from \cite{Grieve:R-R:abVars}.  Indeed, they establish more general versions of those results which apply to simple semihomogeneous vector bundles.  They also refine \cite[Proposition 2.1]{Grieve-cup-prod-ab-var}.

Before stating our main results that are in the direction of continuous Castelnuovo-Mumford regularity, see Theorem \ref{reduced:norm:cont:reg} below, we formulate, in Theorem \ref{R-R-index-lb}, a form of the Riemann-Roch Theorem for line bundles on a given Abelian variety.  It collects results from \cite{Grieve:R-R:abVars} and 
encompasses the traditional Riemann-Roch and Index Theorems for line bundles on Abelian varieties.  It is proved by combining \cite[Theorems 4.1 and 4.4]{Grieve:R-R:abVars}.  

Recall, that the classical formulation of these results were given by Mumford \cite{Mum:v1}, with subsequent refinements by Kempf and Ramanujam \cite{Mum:Quad:Eqns}.  We describe precisely, in Section \ref{Endomorphism:algebra:notation}, the normalized polynomial \eqref{normalized:polynomial:function} which arises in the statement of Theorem \ref{R-R-index-lb}.   As indicated there, this normalized polynomial, \eqref{normalized:polynomial:function}, reflects the structure of the Wedderburn decomposition of the Abelian variety.

In this article, if 
$$D \in \operatorname{NS}_{\QQ}(A)$$ 
is a rational divisor class on a $g$-dimensional Abelian variety $A$, then we define its \emph{index} $\operatorname{i}(D)$ to be the number of positive roots, counted with multiplicities, of any, and in fact all, of its Hilbert polynomials
\begin{equation}\label{Hilb:poly}
\chi(N\lambda + D) := \frac{(N\lambda + D)^g}{ g! } \text{.}
\end{equation}
Here, $\lambda$ is an ample divisor class on $A$ and 
$$g := \dim A $$
is the dimension of $A$.

We refer to \cite[p. 156]{Mum:v1} and \cite[Appendix Theorem 2]{Mum:Quad:Eqns} for more details about the fact that the index $\operatorname{i}(D)$, as defined here, is well-defined (i.e., independent of the choice of polarization $\lambda$).  In particular, to establish that the index of rational divisor classes is well-defined, one first reduces to the case of integral divisor classes.

Note that our concept of index here, differs, in general, from that which is defined in \cite[Section 2.1]{Grieve-cup-prod-ab-var}.  Moreover, it should not be confused with the related notion of what we call \emph{weak index} and which we denote by $\operatorname{j}(D)$.  This concept of weak index is defined for integral divisor classes
$$D \in \operatorname{NS}(A) \text{,}$$ 
and is often times referred to as the \emph{index} in the context of the Generic Vanishing Theory.
(See Section \ref{GV:theory:prelim}.)

For the case of nondegenerate line bundles, i.e., when
$$
\chi(D) \not = 0 \text{,}
$$
the concepts of index and weak index coincide.  (Theorem 
\ref{semi:homog:index:thm} establishes a more general instance of this fact.)

Before stating our main results, especially Theorem \ref{reduced:norm:cont:reg}, recall, that fixing a polarization $\lambda$ on $A$, if
$$
r_\lambda \colon \End^0(A) \rightarrow \End^0(A)
$$
denotes the Rostai involution and
$$
\End_{\lambda}^0(A) := \{ \alpha \in \End^0(A) : r_\lambda(\alpha) = \alpha  \} \text{,}
$$
then there is an induced isomorphism
$$
\NS_{\QQ}(A) \simeq \End^0_{\lambda}(A) \text{.}
$$
We refer to Section \ref{endomorphism:algebras} for further details.

Moreover, recall that the index, as defined here, arises in the Riemann-Roch Theorem.  This is the content of Theorem \ref{R-R-index-lb} which we deduce from \cite{Grieve:R-R:abVars}.

\begin{theorem}\label{R-R-index-lb} Suppose that 
$$A = A_1^{r_1}\times \dots \times A_k^{r_k}$$ 
is an Abelian variety with $A_i$ simple and pairwise nonisogenous Abelian varieties.  Fix an ample divisor class $\lambda$ on $A$.  Put 
$$g := \dim A\text{.}$$ 
Suppose that 
$$f \colon B \rightarrow A$$ 
is an isogeny from a given Abelian variety $B$.  Let 
$$
D \in \operatorname{NS}_{\QQ}(B) := \operatorname{NS}(B) \otimes_{\ZZ} \QQ
$$ 
be a rational divisor class on $B$.  Then, within this context, there exists a normalized polynomial function
\begin{equation}\label{normalized:polynomial:function}
\operatorname{pNrd}_{\lambda}(\cdot) \colon \operatorname{End}_\lambda^0(A) \rightarrow \QQ
\end{equation}
so that the following two assertions hold true.
\begin{enumerate}
\item[(i)]{
If 
$$\alpha = \Phi_{f^*(\lambda)}(D)$$ is the image of 
$$[D] \in \End^0_{f^* \lambda}(B)\text{,}$$ 
in $\End^0_\lambda(A)$, under the homomorphism that is induced by $f$, then 
$$
\chi(D) = \frac{ (D^g)}{ g!} = \sqrt{\operatorname{deg} \phi_{f^* \lambda}} \operatorname{pNrd}_{ \lambda}(\alpha) \text{.}
$$
}
\item[(ii)]{
The \emph{index} $\operatorname{i}(D)$ of $D$ is equal to the number of positive roots, counted with multiplicity, of the polynomial
$$
p_{D,f^*\lambda}(N) := \operatorname{pNrd}_\lambda(N \operatorname{id}_A + \alpha) \text{.}
$$
}
\item[(iii)]{
In particular, if $\mathcal{L}$ is a line bundle on $B$ with class
$$D \in \operatorname{NS}(B)\text{,}$$ 
then
$$
\H^j(B,\mathcal{L}) = 0
$$
for 
$$0\leq j < \text{the number of $p_{D,f^*\lambda}(N)$'s positive roots}$$ 
and
$$
\H^{g-j}(B,\mathcal{L}) = 0
$$
for 
$$0 \leq j < \text{the number of $p_{D,f^*\lambda}(N)$'s negative roots.}$$
}
\end{enumerate}
\end{theorem}

In this article, we define the \emph{index} $\operatorname{i}(\mathcal{E})$, 
for a semihomogeneous vector bundle $\mathcal{E}$ on $A$, to be the number of positive roots, counted with multiplicity, of any, and, by Lemma \ref{Index:Well:Defined:Lemma}, all of its Hilbert polynomials.  Again, this concept of index differs, in general, from the concept of index that is defined in \cite[Section 2.1]{Grieve-cup-prod-ab-var}.

Theorem \ref{R-R-index-lb} has the following consequence for the index of simple semihomogeneous vector bundles.

\begin{corollary}\label{Semi:homog:RR}
In the setting of Theorem \ref{R-R-index-lb}, suppose that 
$$
\alpha \in \operatorname{End}^0_\lambda(A)
$$ 
is the image of 
$$
[\operatorname{det}(\mathcal{E})] \in \operatorname{End}^0_{f^*\lambda}(B) \text{,}
$$
for $\mathcal{E}$ a simple semihomogeneous vector bundle on $B$.  Then,  $\operatorname{i}(\mathcal{E})$,  the index of $\mathcal{E}$, is equal to the number of positive roots counted with multiplicity of the polynomial
$$
\mathrm{p}_{\operatorname{det}(\mathcal{E}),f^*(\lambda)}(N) := \mathrm{pNrd}_{\lambda}(N \operatorname{id}_A + \alpha) \text{.}
$$
\end{corollary}

As an additional application of Theorem \ref{R-R-index-lb}, here we use it to build on \cite[Theorem A]{Kuronya:Mustopa:2020}.  In doing so, we establish the following result.

\begin{theorem}\label{reduced:norm:cont:reg}  Suppose that 
$$A = A_1^{r_1}\times \dots \times A_k^{r_k}$$ is an Abelian variety with $A_i$ simple and pairwise nonisogenous Abelian varieties.  Fix an ample divisor class $\lambda$ on $A$.
Let 
$$f \colon B \rightarrow A$$ 
be an isogeny and  let 
$$
\operatorname{pNrd}_{\lambda}(\cdot) \colon \operatorname{End}_\lambda^0(A) \rightarrow \QQ
$$
be the normalized polynomial function that is given by Theorem \ref{R-R-index-lb} (see Equation \eqref{normalized:polynomial:function}).  Fix a rational divisor class, in $\NS_{\QQ}(B)$, corresponding to 
$$
\gamma \in \End^0_{f^*\lambda}(B) \text{,}
$$
{under the identification
$$\NS_{\QQ}(B) \simeq \End^0_{f^*\lambda}(B)\text{.}$$
 }
Let $\alpha$ be its image in $\End^0_\lambda(A)$ under the natural map
$$
\Phi_{f^*(\lambda)}(\cdot) \colon \End^0_{f^*\lambda}(B) \rightarrow \End^0_{\lambda}(A) \text{.}
$$
Suppose that $m \in \ZZ$ is the smallest integer for which 
$$
\operatorname{pNrd}_\lambda((m-i)\id_A + \alpha) = 0
$$
or for which the polynomial
$$
\operatorname{pNrd}_\lambda((N+m-i)\id_A + \alpha)
$$
fails to have $i$ positive roots (counted with multiplicities) for all 
$$i \in \{1,\dots,g\}\text{.}$$  
Finally, 
let $\mathcal{E}$ be a semihomogeneous vector bundle on $B$ and assume that the class 
$$
\frac{\det(\mathcal{E})}{\operatorname{rank}(\mathcal{E})} \in \NS_{\QQ}(B) 
$$
is identified with $\gamma$.
Then $\operatorname{reg}_{\operatorname{cont}}(\mathcal{E},f^*\lambda)$, the \emph{continuous Castelnuovo-Mumford regularity} of $\mathcal{E}$ with respect to the polarization $f^*\lambda$, is equal to $m$.
\end{theorem}

We prove Theorem \ref{reduced:norm:cont:reg} in Section \ref{proof:reduced:norm:cont:reg}. It is established by first reducing to the case of simple semihomogeneous vector bundles.  In Sections \ref{semihomog:index:thm} and \ref{GV:theory:prelim}, we explain how these matters relate to the Index and Generic Vanishing Theorems for simple semihomogeneous vector bundles.  

In doing so, we complement related results of Gulbrandsen, \cite[Propositions 5.1 and 6.3]{Gulbrandsen:2008}, and Kempf, \cite[Appendix]{Mum:Quad:Eqns}.  In Section \ref{Abelian:variety:notation}, we establish our notation for Abelian varieties and, in Section  \ref{Endomorphism:algebra:notation}, we fix our notation for endomorphism algebras and the Riemann-Roch Theorem. 

Our results here, together with those of our earlier work \cite{Grieve:R-R:abVars}, indicate, in a precise way the manner in which the Wedderburn decomposition of a given Abelian variety is reflected in cohomological and global generation properties of (higher rank) semihomogeneous vector bundles.  This picture expands upon earlier work of Mumford \cite{Mum:v1}.   

As one direction for future investigation, it remains an interesting problem to have an explicit knowledge of the normalized polynomials $\operatorname{pNrd}_{\lambda}(\cdot)$, which arise in the statement of Theorem \ref{reduced:norm:cont:reg}, for a wide class of Abelian varieties. For that it is necessary to have a detailed understanding of the endomorphism algebra together with the Rosati involution.  

To help place matters into perspective, note that already it is an interesting question to determine which integers can be realized as Picard numbers of Abelian varieties.  We refer to the article \cite{Hulek:Laface:2019}, as one more recent work in that direction.

\subsection*{Acknowledgements} This work benefited from the Banff International Research Station workshops 19w5164 \emph{Interactions between Brauer Groups, Derived Categories and Birational Geometry of Projective Varieties} and 20w5176 \emph{Derived, Birational, and Categorical Algebraic Geometry (Online)}, in addition to several other virtual seminars and conferences during the Spring, Summer and Fall 2020 semesters.  I thank colleagues for their interest, inspiration and lectures.  Further, I thank the Natural Sciences and Engineering Research Council of Canada for their support through my grants DGECR-2021-00218 and RGPIN-2021-03821. Finally, I thank anonymous referees for their careful reading of this article and for offering several helpful comments and suggestions.   

\section{Notation for Abelian varieties}\label{Abelian:variety:notation}

In what follows, we work over a fixed algebraically closed base field $\kk$.  If $A$ is an Abelian variety, then $\hat{A}$ denotes the dual Abelian variety and $\mathcal{P}$ denotes the normalized Poincar\'{e} line bundle on 
$A \times \hat{A}\text{.}$  If 
$$x \in A\text{,}$$ 
then 
$$\tau_x \colon A \rightarrow A$$ 
denotes translation by $x$ in the group law.   

If 
$$\hat{x} \in \hat{A}\text{,}$$ 
then $\mathcal{P}_{\hat{x}}$ denotes the translation invariant line bundle on $A$ that is determined by $\hat{x}$.  

Multiplication in the group law is denoted as 
$$m \colon A \times A \rightarrow A\text{.}$$  
The projections of $A \times A$ onto the first and second factors, respectively, are denoted by $p_1$ and $p_2$.
  On the other hand, the projections of $A \times \hat{A}$ onto the first and second factors, respectively, are denoted by $p_A$ and $p_{\hat{A}}$. 

The endomorphism algebra of $A$ is denoted by $\operatorname{End}(A)$.  We also put 
$$\operatorname{End}^0(A) := \operatorname{End}(A) \otimes_{\ZZ} \QQ\text{.}$$  
We let $\operatorname{NS}(A)$ denote the N\'{e}ron-Severi group of $A$ and put
$$
\operatorname{NS}_{\QQ}(A) := \operatorname{NS}(A)\otimes_{\ZZ}\QQ \text{.}
$$
When no confusion is likely, at times we use the same notation to denote the class, in $\NS(A)$, of a divisor $D$ on $A$.  By a rational divisor class, we mean an element of $\NS_\QQ(A)$.  

Moreover, by the \emph{Euler characteristic} of a rational divisor class 
$$D \in \NS_{\QQ}(A)\text{,}$$
is meant the quantity 
$$
\chi(D) = (D^g)/g! \text{.}
$$ 
Here, $g$ is the dimension of $A$ and $(D^g)$ is $D$'s $g$-fold self-intersection number.

If 
$$\alpha \in \End^0(A)\text{,}$$ 
then 
$$\hat{\alpha} \in \End^0(\hat{A})$$ 
is its dual.  If $D$ is a divisor on $A$, then 
$$\phi_D \colon A \rightarrow \hat{A}$$ 
is the homomorphism that is defined by
$$x \mapsto \tau_x^*\Osh_A(D) \otimes \Osh_A(-D) \text{.}$$

We refer to \cite{Mum:v1} for more details about Abelian varieties.  Much of the most basic theory is summarized in \cite[Section 2]{Grieve:R-R:abVars}.

\section{Albert algebras}\label{endomorphism:algebras}

Recall, that an \emph{Albert algebra} consists of a division algebra $\Delta$, of finite dimension over $\QQ$, together with an involution 
$$' \colon \Delta \rightarrow \Delta\text{,}$$ 
which we denote as
$$\alpha \mapsto \alpha'\text{,}$$
and 
which is \emph{positive} in the sense that 
$$\operatorname{Trd}_{\Delta / \QQ}(\alpha \alpha') > 0$$ 
for all 
$$0 \not = \alpha \in \Delta\text{.}$$  
Here, 
$$\operatorname{Trd}_{\Delta / \QQ}(\cdot) \colon \Delta \rightarrow \QQ
$$ denotes the reduced trace from $\Delta$ to $\QQ$.  Similarly, $$\operatorname{Nrd}_{\Delta / \QQ}(\cdot)\colon \Delta \rightarrow \QQ$$ 
denotes the reduced norm in what follows.

If $A$ is a simple Abelian variety, then each ample divisor class $\lambda$ induces a positive involution  $r_\lambda$ on the division algebra $\operatorname{End}^0(A)$.  Especially, for each fixed polarization, the pair $(\operatorname{End}^0(A),r_{\lambda})$ is an Albert algebra.  We recall the definition of the involution $r_\lambda$ in Section \ref{Endomorphism:algebra:notation}.

We refer to \cite[Theorem 5.1]{Grieve:R-R:abVars}, and the references therein, for a statement and more detailed discussion of Albert's theorem.  Several examples of Albert's theorem, as it pertains to endomorphism algebras and Rosati involutions for isotypic Abelian varieties and products thereof, are given in \cite[Sections 6 and 7]{Grieve:R-R:abVars}. 

\section{Notation for endomorphism algebras}\label{Endomorphism:algebra:notation}

For later use, we fix notation for endomorphism algebras of Abelian varieties.  We follow the conventions of \cite{Grieve:R-R:abVars} closely.  First of all, consider the case of an Abelian variety $A$ of the form
$$
A := A_1^{r_1} \times \dots \times A_k^{r_k} \text{,}
$$
where each of the $A_i$ are simple and pairwise nonisogeneous Abelian varieties.  

In terms of the endomorphism algebra
$$
R := \End^0(A) \text{,}
$$
we may write
$$
R_i := M_{r_i}(\Delta_i) \text{.}
$$
Let $Z_i$ be the centre of $\Delta_i$ and let
$$
\Nrd_{R_i / \QQ}(\cdot) \colon R_i \rightarrow \QQ
$$
be the reduced norm.

Put
$$
g_i := \dim A_i \text{, } m_i^2 := \dim_{Z_i} \Delta_i \text{ and } t_i := [Z_i : \QQ] \text{, }
$$
for $i = 1,\dots, k$.

Fix an ample divisor $\lambda$ on $A$.  Recall, that the \emph{Rosati involution}
$$
r_\lambda \colon R \rightarrow R
$$
is defined by the condition that
$$
\alpha \mapsto \phi_{\lambda}^{-1} \circ \hat{\alpha} \circ \phi_\lambda \text{.}
$$
Moreover, the map
$$
D \mapsto \phi^{-1}_{\lambda} \circ \phi_D
$$
induces an isomorphism
$$
\Phi_\lambda \colon \NS_{\QQ}(A) \xrightarrow{\sim} \End^0_\lambda(A) \text{.}
$$
Here, $\phi_\lambda^{-1}$ is the inverse of $\phi_\lambda$ and
$$
\End^0_\lambda(A) := \{\alpha \in \End^0(A) : \alpha = r_{\lambda}(\alpha) \} \text{.}
$$

Within the present context, as noted in \cite[Corollary 3.7]{Grieve:R-R:abVars}, the function
$$
\prod_{i=1}^k \Nrd_{R_i / \QQ}(\cdot)^{2g_i / (t_i m_i) }|_{\End_\lambda^0(A)} \colon \End^0_{\lambda}(A) \rightarrow \QQ
$$
is the square of a rational valued homogeneous polynomial function of degree $g$ on $\End^0_{\lambda}(A)$.  
It is normalized so as to take value $1$ on 
$$\operatorname{id}_A = 1_R$$ and is denoted by $\operatorname{pNrd}_{\lambda}(\cdot)$.

In particular, if $D$ is a divisor on $A$, then, as in \cite[Theorem 4.1 (a)]{Grieve:R-R:abVars}, the Riemann-Roch Theorem may be expressed as
$$
\chi(D) = \frac{(D^g)}{g!} = \sqrt{\deg \phi_\lambda} \operatorname{pNrd}_{\lambda}(\Phi_\lambda(D)) \text{.}
$$

More generally, let 
$$f \colon B \rightarrow A$$ 
be an isogeny and $D$ a divisor on $B$.  Let 
$$\alpha \in \End^0_\lambda(A)$$ 
correspond to $D$ under the induced map
$$
\End^0_{f^* \lambda}(B) \rightarrow \End^0_\lambda(A) \text{.}
$$
Then, as in \cite[Theorem 4.1 (b)]{Grieve:R-R:abVars}, the Riemann-Roch theorem for divisors $D$ on $B$ can be expressed as
\begin{equation}
\chi(D) = \frac{(D^g)}{g!} = \sqrt{\deg \phi_{\beta^* \lambda}} \operatorname{pNrd}_{\lambda}(\alpha) \text{.}
\end{equation}

\section{Semihomogeneous vector bundles and the index theorem}\label{semihomog:index:thm}

The concept of \emph{semihomogeneous vector bundle} over a given Abelian variety $A$ is due to Mukai \cite{Muk78}, building on earlier work of Mumford \cite{Mum:v1} and Atiyah \cite{Atiyah}.  Specifically, a vector bundle $\mathcal{E}$ over $A$ is called \emph{semihomogeneous} if for all 
$x \in A\text{,}$
$$\tau^*_x \mathcal{E} \simeq \mathcal{E} \otimes \mathcal{L}$$ 
for some line bundle $\mathcal{L}$ over $A$.

A vector bundle $\mathcal{E}$ on $A$ is called \emph{simple} if 
$$\operatorname{End}_{\Osh_A}(\mathcal{E}) \simeq \kk \text{.}$$
The \emph{simple semihomogeneous} vector bundles over $A$ were characterized by Mukai \cite[Theorem 5.8]{Muk78}.    In particular, for a simple vector bundle $\mathcal{E}$ on $A$, the following four conditions are equivalent

\begin{itemize}
\item{
$\dim_{\kk} \H^1(A,\mathcal{E}nd_{\Osh_A}(\mathcal{E})) = g$;
}
\item{
$\mathcal{E}$ is semihomogenous;
}
\item{
$\mathcal{E}nd_{\Osh_A}(\mathcal{E})$ is a homogeneous vector bundle i.e., 
$$\tau^*_x \mathcal{E}nd_{\Osh_A}(\mathcal{E}) \simeq \mathcal{E}nd_{\Osh_A}(\mathcal{E})\text{,}$$ for all $x \in A$; and
}
\item{
there exists an isogeny 
$$f \colon B \rightarrow A$$ 
and a line bundle $\mathcal{L}$ on $B$ so that 
$$\mathcal{E} \simeq f_*(\mathcal{L})\text{.}$$
}
\end{itemize}

We refer to \cite{Grieve:theta} for the theory of theta groups that are associated to simple semihomogeneous vector bundles and to \cite{Grieve:biextensions}, building on work of Brion \cite{Bri}, for the main results about homogeneous Severi-Brauer varieties.

For a simple vector bundle $\mathcal{E}$ over $A$, let $\K(\mathcal{E})$ be the maximal subscheme of $A$ over which $m^*(\mathcal{E})$ and $p_{A}^*(\mathcal{E})$ are isomorphic \cite[Definition 3.8]{Muk78}.  
Moreover, recall that if $\mathcal{E}$ is a rank $r$ semihomogeneous vector bundle over $A$, then
$$
\chi(\mathcal{E}) = \frac{\chi(\operatorname{det}(\mathcal{E}))}{r^{g-1}} \text{,}
$$
\cite[Proposition 6.12]{Muk78}.

Further, if $\mathcal{E}
$ is a simple semihomogeneous vector bundle, then 
\begin{itemize}
\item{$\dim \mathrm{K}(\mathcal{E}) = \dim \mathrm{K}(\operatorname{det}(\mathcal{E}))$; and}
\item{
if $\chi(\mathcal{E}) \not = 0$, then $\operatorname{ord}(\mathrm{K}(\mathcal{E})) = \chi(\mathcal{E})^2$;
}
\end{itemize}
see  \cite[Corollary 7.9]{Muk78}.

In what follows, we say that a vector bundle $\mathcal{E}$ on $A$ is \emph{nondegenerate} if its \emph{Euler characteristic}
$$
\chi(\mathcal{E}) := \sum_{i=0}^g (-1)^i \dim_{\kk} \H^i(A,\mathcal{E})
$$
is nonzero.  If 
$$\chi(\mathcal{E}) = 0\text{,}$$ 
then $\mathcal{E}$ is called \emph{degenerate}.   

Let $\mathcal{E}$ be a vector bundle on $A$.
Fixing an ample line bundle $\Osh(1)$ on $A$, consider the Hilbert polynomial 
\begin{equation}\label{vb:hilb:poly}
\operatorname{HilbPoly}_{\Osh(1)}(\mathcal{E}) := \chi(\mathcal{E}(N)) \text{.}
\end{equation}

In Definition \ref{semi:homog:index:well:defined} below, we use Hilbert polynomials to define a concept of index for simple semihomogeneous vector bundles.  The following lemma implies that this definition is indeed well-defined.  We include a proof for completeness.

\begin{lemma}\label{Index:Well:Defined:Lemma}
Assume that $\mathcal{E}$ is a simple semihomogeneous vector bundle on $A$.  Then, the roots of the polynomial \eqref{vb:hilb:poly} are real, for all ample line bundles $\Osh(1)$ on $A$.  Moreover, the number of positive roots, counted with multiplicity, of those polynomials is well-defined.  In particular, it is independent of the choice of ample line bundle.
\end{lemma}

\begin{proof}
By assumption, $\mathcal{E}$ is a simple semihomogeneous vector bundle.  In particular, there exists an isogeny
$$
f \colon B \rightarrow A
$$
together with a line bundle $\mathcal{L}$ on $B$, which has the property that
$$
\mathcal{E} \simeq f_*(\mathcal{L}) \text{.}
$$
Fix an ample line bundle $\Osh_A(1)$ on $A$ and put
$$
\Osh_B(1) := f^* \Osh_A(1)  \text{.}
$$
Then, $\Osh_B(1)$ is an ample line bundle on $B$.  Moreover, if $N \in \ZZ$, then the projection formula implies that
\begin{align*}
\mathcal{E}(N) & := \mathcal{E} \otimes \Osh_A(N) \\
& \simeq f_* (\mathcal{L} \otimes \Osh_B(N)) \text{.}
\end{align*}

Further, since $f$ is an isogeny
$$
\R^i f_*\left( \mathcal{L} \otimes \Osh_B(N) \right) = 0 \text{,}
$$
for $i > 0$.  In particular, the Leray spectral sequence implies equality of Euler  characteristics
$$
\chi(A,f_*(\mathcal{L} \otimes \Osh_B(N)) = \chi(B,\mathcal{L} \otimes \Osh_B(N)) \text{.}
$$
Hence
\begin{equation}\label{index:well:defined:relation}
\chi(A,\mathcal{E}(N)) = \chi(B,\mathcal{L} \otimes \Osh_B(N)) \text{.}
\end{equation}

Now,  $\mathcal{L}$ is a line bundle on $B$ and it is known,  by \cite[Appendix Theorem 2]{Mum:Quad:Eqns}, for example, that the roots of all of the Hilbert polynomials of $\mathcal{L}$ are real and that the number of positive roots, counted  with multiplicities of such polynomials is independent of the choice of ample line bundle on $B$.

Finally, since  each ample line bundle on $A$ pulls back to an ample line bundle on $B$, the conclusion of the lemma follows from the above relation \eqref{index:well:defined:relation}.
\end{proof}

\begin{defn}\label{semi:homog:index:well:defined}
If $\mathcal{E}$ is a simple semihomogeneous vector bundle on $A$, then let $\operatorname{i}(\mathcal{E})$ be the number of positive (real) roots of any, and hence all, of the Hilbert polynomials \eqref{vb:hilb:poly}.
Define $\operatorname{i}(\mathcal{E})$ to be the number of positive real roots counted with multiplicity.  In what follows, we say that $\operatorname{i}(\mathcal{E})$ is the \emph{index} of $\mathcal{E}$.
\end{defn}

\begin{remark}
Note that the concept of index $\operatorname{i}(\mathcal{E})$, for a given semihomogeneous vector bundle $\mathcal{E}$, as we have defined here, differs, in general, from the concept of index that is defined in \cite[Section 2.1]{Grieve-cup-prod-ab-var}.  But, for the case of nondegenerate semisimple vector bundles, for example, these concepts of index are equivalent.
\end{remark}

In Proposition \ref{line:bundle:index} below, we collect, from \cite[Appendix]{Mum:Quad:Eqns} and \cite[Proposition 5.1]{Gulbrandsen:2008}, a number of cohomological properties of line bundles.  They are used in our proof of Theorem \ref{semi:homog:index:thm}.  
We illustrate the conclusion of Proposition \ref{line:bundle:index}, in Example \ref{weak:index:E:E:example}.  

\begin{proposition}\label{line:bundle:index}  
Let $\mathcal{L}$ be a line bundle on a $g$-dimensional Abelian variety $A$.  Let $\K^0(\mathcal{L})$ be the identity component, with the reduced subscheme structure, of the group scheme $\K(\mathcal{L})$.  Let 
$$f \colon A \rightarrow A / \K^0(\mathcal{L})$$ 
be the canonical homomorphism.  Then, the following assertions hold true.
\begin{enumerate}
\item[(i)]{
If $\mathcal{L}|_{\K^0(\mathcal{L})}$ is nontrivial, then
$$
\H^\ell(A,\mathcal{L}) = 0 \text{,}
$$
for all $\ell = 0,\dots,g$.  Further
$$
\R^\ell {p_{\hat{A}}}_* \left( p_A^* \mathcal{L} \otimes \mathcal{P} \right) = 0 \text{,}
$$
for all $\ell = 0,\dots,g$.
}
\item[(ii)]{
If $\mathcal{L}|_{\K^0(\mathcal{L})}$ is trivial, then
$$
\H^\ell(A,\mathcal{L}) \not = 0
$$
if and only if
$$
\ell \in [\operatorname{i}(\mathcal{L}), \operatorname{i}(\mathcal{L}) + \dim \K(\mathcal{L}) ] \text{.}
$$
Moreover,
$$
\R^\ell {p_{\hat{A}}}_* \left(p_A^* \mathcal{L} \otimes \mathcal{P} \right) = 0
$$
for 
$$
\ell \not = \dim \K(\mathcal{L}) + \operatorname{i}(\mathcal{L}) \text{.}
$$
}
\end{enumerate}
\end{proposition}

\begin{proof}
 For (i), if $\mathcal{L}|_{\K^0(\mathcal{L})}$ is nontrivial, then \cite[Appendix Lemma 1]{Mum:Quad:Eqns} implies that $\tau_x^* \mathcal{L}|_{\K^0(\tau^*_x \mathcal{L}) }$ is nontrivial for all $x \in A$.  Hence, by \cite[Appendix Theorem 1 (i)]{Mum:Quad:Eqns}, 
$$
\H^\ell(A,\tau^*_x \mathcal{L}) = \H^\ell(A,\mathcal{L} \otimes \mathcal{P}_{\hat{x}}) = 0 \text{,}
$$
for all $\ell = 0,\dots,g$ and all $x \in A$.  Thus
$$
\R^\ell{p_{\hat{A}}}_* \left( p_A^*\mathcal{L} \otimes \mathcal{P} \right) = 0 \text{,}
$$
for all $\ell = 0,\dots,g$.

For (ii), if $\mathcal{L}|_{\K^0(\mathcal{L})}$ is trivial, then, as noted in \cite[Appendix Theorem 1 (ii)]{Mum:Quad:Eqns}, there exists a nondegenerate line bundle $\mathcal{M}$ on $A / \K^0(\mathcal{L})$ and a point $\hat{x} \in \hat{A}$, which has the two properties that
\begin{itemize}
\item{
$
\mathcal{L} \simeq f^*(\mathcal{M}) \otimes \mathcal{P}_{\hat{x}} \text{; and}
$}
\item{
$\H^\ell(A,\mathcal{L}) \simeq \H^{\operatorname{i}(\mathcal{M})}(A/\K^0(\mathcal{L}),\mathcal{M}) \otimes \H^{\ell - \operatorname{i}(\mathcal{M})}(\K^0(\mathcal{L}), \Osh_{\K^0(\mathcal{L})} )$.
}
\end{itemize}
Moreover, it is known, see \cite[p. 100]{Mum:Quad:Eqns}, that
$$
\operatorname{i}(\mathcal{L}) = \operatorname{i}(\mathcal{M})\text{.}
$$
In particular, the above considerations imply that 
$$
\H^\ell(A,\mathcal{L}) \not = 0
$$
if and only if
$$
\ell \in [\operatorname{i}(\mathcal{L}), \operatorname{i}(\mathcal{L}) + \dim \K(\mathcal{L}) ] \text{.}
$$
It remains to establish vanishing of the higher direct image sheaves 
$$
\R^\ell {p_{\hat{A}}}_* \left(p_A^* \mathcal{L} \otimes \mathcal{P} \right) \text{,}
$$
for 
$$\ell \not = \operatorname{i}(\mathcal{L}) + \dim \K(\mathcal{L})\text{.}$$
To this end, we argue as in \cite[Proof of Proposition 5.1]{Gulbrandsen:2008}.  

To begin with, 
there is a natural isomorphism 
$$
\mathbf{R}{p_{\hat{A}}}_*\left( p_A^*((-)\otimes\mathcal{P}_{\hat{x}})\otimes\mathcal{P}\right) \simeq \tau_{\hat{x}}^* \mathbf{R}{p_{\hat{A}}}_*\left( p_A^*(-)\otimes\mathcal{P}\right) \text{,}
$$
of functors from $\mathbf{D}(A)$ to $\mathbf{D}(\hat{A})$. 
Here, $\mathbf{D}(A)$ and $\mathbf{D}(\hat{A})$ are the respective bounded derived categories of coherent sheaves on $A$ and $\hat{A}$.

Let 
$$
d := \dim \mathrm{K}(\mathcal{L})
$$
and put 
$$
B := A/\K^0(\mathcal{L}) \text{.}
$$
Let $\mathcal{Q}$ be the normalized Poincar\'{e} line bundle on $B \times \hat{B}$.

Then, since 
$$
f \colon A \rightarrow B
$$ 
is flat and since the dual morphism 
$$ 
\hat{f} \colon \hat{B} \rightarrow \hat{A} 
$$
is finite, 
it follows that for all integers $\ell$, 
\begin{align*}
\R^\ell{p_{\hat{A}}}_*(p_A^*f^*(\mathcal{M} \otimes \mathcal{P}_{\hat{x}} ) \otimes  \mathcal{P}) 
& 
\simeq \tau^*_{\hat{x}}(\R^\ell{p_{\hat{A}}}_*(p_A^*f^*(\mathcal{M}) \otimes \mathcal{P} ))  
\\
& 
\simeq \tau^*_{\hat{x}}(\hat{f}_*(\R^{\ell-d}{p_{\hat{B}}}_*(p_B^*\mathcal{M} \otimes \mathcal{Q} )) )\text{.}
\end{align*}
Since $\mathcal{M}$ is nondegenerate, it has a unique nonzero cohomology group $\H^{\operatorname{i}(\mathcal{M})}(B,\mathcal{M})$.
 The claim then follows.
\end{proof}

Finally, we conclude this section by mentioning an observation from \cite[Section 4]{Kuronya:Mustopa:2020}, which pertains to the nature of semihomogeneous vector bundles, and their continuous Castelnouvo-Mumford regularity, on products of nonisogenous Abelian varieties with Picard number one.

\begin{example}
\label{non:isog:pic:one:products}  Consider the case that 
$$A = A_1 \times A_2$$ 
is a product of nonisogenous Abelian varieties, each of which have Picard number equal to one.  Then, since $A_1$ and $A_2$ are not isogeneous and each have Picard number one, it follows that $A$ has Picard number two (see for instance \cite[Proposition 6.2]{Grieve:R-R:abVars}).  Moreover, the N\'{e}ron-Severi group $\NS_{\QQ}(A)$ is generated by the classes of $p_{A_i}^*\lambda_i$, for $i = 1,2$, where $\lambda_1$ and $\lambda_2$ are ample line bundles on $A_1$ and $A_2$, respectively.  

Now, suppose that $\mathcal{E}$ is a rank $r$ semihomogeneous vector bundle on $A$ with restrictions $\mathcal{E}_1$ and $\mathcal{E}_2$ to $A_1 \times \{0\}$ and $\{0\} \times A_2$, respectively.  Then since $\mathcal{E}$ is a semihomogeneous vector bundle on $A$, it follows easily that $\mathcal{E}_1 \boxtimes \mathcal{E
}_2$ is a semihomogeneous vector bundle on $A$.  

Further, since
$$
\det(\mathcal{E}_1 \boxtimes \mathcal{E}_2) = r \cdot p_{A_1}^* \det(\mathcal{E}_1) + r \cdot p_{A_2}^* \det(\mathcal{E}_2)\text{,}
$$
and  
$$\operatorname{det}(\mathcal{E}) = p_{A_1}^* \det(\mathcal{E}_1) + p_{A_2}^* \det(\mathcal{E}_2) \text{,}$$
as classes in $\NS(A)$, it follows that  
$$
\frac{\operatorname{det}(\mathcal{E})}{r} = \frac{\operatorname{det}(\mathcal{E}_1 \boxtimes \mathcal{E}_2)}{r^2} \text{,}
$$
as classes in $\operatorname{NS}_{\QQ}(A)$.  Finally, Theorem 1.3 (i) implies that if $\Osh(1)$ is an ample line bundle on $A$, then
$$
\operatorname{reg}_{\mathrm{cont}}(\mathcal{E},\Osh(1)) = \operatorname{reg}_{\mathrm{cont}}(\mathcal{E}_1 \boxtimes \mathcal{E}_2,\Osh(1)) \text{.}
$$
\end{example}

\section{Generic vanishing theory}\label{GV:theory:prelim}

In this section, we recall the most basic concepts from the \emph{Generic Vanishing Theory}.  We mostly follow \cite[Section 2]{Pareschi:Popa:2003}.
Let $\mathcal{F}$ be a coherent sheaf over a $g$-dimensional Abelian variety $A$.  Fix a nonnegative integer $i \geq 0$ and put
$$
\operatorname{V}^i(A,\mathcal{F}) := \{ \hat{x} \in \hat{A} : \H^i(A,\mathcal{F} \otimes \mathcal{P}_{\hat{x}}) \not = 0 \} \text{.}
$$
This is the \emph{$i$th cohomological support locus of $\mathcal{F}$}.  If 
$$
\operatorname{codim}(\operatorname{V}^i(A,\mathcal{F})) > i
$$
for all $i > 0$, then $\mathcal{F}$ is called \emph{Mukai-regular}.  If
$$
\operatorname{codim}(\operatorname{V}^i(A,\mathcal{F})) \geq i
$$
for all $i>0$, then $\mathcal{F}$ is called a \emph{Generic Vanishing sheaf}.  For example, if $\mathcal{F}$ is Mukai-regular, then it is a Generic Vanishing sheaf.

Now, returning to more general considerations, suppose that  for all  nonempty Zariski open subsets 
$$U \subseteq \hat{A}\text{,}$$ 
 the evaluation map
$$
\bigoplus_{\hat{x} \in U} \H^0(A,\mathcal{F} \otimes \mathcal{P}_{\hat{x}}) \otimes \mathcal{P}^{-1}_{\hat{x}} \rightarrow \mathcal{F}
$$
is surjective.  Then, within this context, $\mathcal{F}$ is called \emph{continuously globally generated} \cite[Definition 2.10]{Pareschi:Popa:2003} (compare also with \cite[Proposition 2.13]{Pareschi:Popa:2003} and \cite[Definition 1.3]{Kuronya:Mustopa:2020}).  

On the other hand, if
$$
\R^i{p_{\hat{A}*}}\left( p^*_A \mathcal{F} \otimes \mathcal{P}\right) \not = 0
$$
for at most one 
$$i \in \{0,\dots,g\}\text{,}$$ 
then $\mathcal{F}$ is called a \emph{Weak Index Theorem sheaf}.  If
$$
\R^j{p_{\hat{A}*}}(p_A^*\mathcal{F} \otimes \mathcal{P}) = 0
$$
for all but one $j$, then, here, the  \emph{weak index} of $\mathcal{F}$ is defined to be the unique integer 
$$j(\mathcal{F}) \in \{0,\dots,g\}$$ 
for which 
$$
\R^{\mathrm{j}(\mathcal{F})}{p_{\hat{A}*}}(p_A^*\mathcal{F} \otimes \mathcal{P}) \not = 0 \text{.}
$$
If
$$
\operatorname{V}^j(A,\mathcal{F}) \not = \emptyset
$$
for at most one 
$$j \in \{0,\dots,g\}\text{,}$$ 
then $\mathcal{F}$ is called an \emph{Index Theorem sheaf}.  

Equivalently, $\mathcal{F}$ is called an \emph{Index Theorem sheaf} if there exists 
$$j \in \{0,\dots,g\}$$ 
so that
$$
\H^\ell(A,\mathcal{F} \otimes \mathcal{P}_{\hat{x}}) = 0 \text{,}
$$
for all $\ell \not = j$ 
and all $\hat{x} \in \hat{A}\text{.}$

Recall, that if an Index Theorem sheaf $\mathcal{F}$ is nonzero, 
then the cohomology and flat base change theorem, \cite[p.~51]{Mum:v1}, implies that it is a Weak Index Theorem sheaf.

The following example illustrates Proposition \ref{line:bundle:index} and Theorem \ref{semi:homog:index:thm}, in addition to the concepts of Weak Index Theorem and Index Theorem sheaves.  It also provides an illustration of \cite[Appendix Theorem 1 (i)]{Mum:Quad:Eqns}.

\begin{example}\label{weak:index:E:E:example}
Let $E$ be an elliptic curve and consider its self-product
$$A = E \times E.$$ 
Fix a point $x \in A$ and put
$$ \mathcal{L} = \Osh_E \boxtimes \Osh_E(x)\text{.}$$ 
Then
$$\H^0 (A,\mathcal{L}) \simeq \H^0(E,\Osh_E) \otimes \H^0(E,\Osh_E(x)),$$ 
$$\H^1(A,\mathcal{L}) \simeq  \H^1(E,\Osh_E) \otimes \H^0(E, \Osh_E(x))$$ 
and 
$$\H^2(A,\mathcal{L}) = 0.$$
Thus
$$h^0(A,\mathcal{L}) \not = 0 \text{, } h^1(A,\mathcal{L}) \not = 0 \text{ and } h^2(A,\mathcal{L}) = 0 \text{.}$$  
Now observe that if
$$\mathcal{M} := \Osh_E(x)\text{,}$$ 
then
$$
\mathcal{L} = f^* \mathcal{M} \text{,}
$$
for 
$$f = p_2$$ 
the projection of $A$ onto its second factor.  This is the canonical homomorphism 
$$f \colon A \rightarrow A / \K^0(\mathcal{L})\text{.}$$ 
Also
$$\operatorname{i}(\mathcal{M}) = 0 \text{.}$$
Further
$$\dim \K(\mathcal{L}) = 1$$
and
$$
\operatorname{j}(\mathcal{L}) = \dim \K(\mathcal{L}) + \operatorname{i}(\mathcal{M}) = 1\text{.}
$$
Finally
$$
\operatorname{i}(\mathcal{L}) = 0
\text{,}$$
since the Hilbert polynomial 
$$\chi(\mathcal{M}^{\otimes N} \boxtimes \mathcal{M}^{\otimes N} \otimes \mathcal{L}) = N(N+1) $$
 has roots $N = 0$ and $N = -1$.  

This shows that, in general, the \emph{weak index} $\operatorname{j}(\mathcal{L})$ need not equal the \emph{index} $\operatorname{i}(\mathcal{L})$.  On the other hand,
Proposition \ref{line:bundle:index} shows that $\mathcal{L}$ is a Weak Index Theorem sheaf with weak index $\operatorname{j}(\mathcal{L}) = 1$.  But on the other hand, since 
$$h^0(A,\mathcal{L}) \not = 0\text{,}$$ 
$\mathcal{L}$ is not an Index Theorem sheaf.
\end{example}

The preceding discussion, as it applies, in particular, to simple semihomogeneous vector bundles is summarized in Theorem \ref{semi:homog:index:thm} below.  Its statement includes a key aspect to the proof of  \cite[Theorem A]{Kuronya:Mustopa:2020}, for example, namely the fact that the weak index of a simple semihomogeneous vector bundle is the same as that of its determinant line bundle.  Theorem \ref{semi:homog:index:thm} also allows Corollary \ref{Semi:homog:RR} to be deduced from Theorem \ref{R-R-index-lb}.  (See Section \ref{proof:reduced:norm:cont:reg}.)

\begin{theorem}\label{semi:homog:index:thm}
Let $\mathcal{E}$ be a simple semihomogeneous vector bundle on a $g$-dimensional Abelian variety $A$.  
The following assertions hold true.
\begin{enumerate}
\item[(i)]{
If $\Osh(1)$ is an ample line bundle on $A$, then the roots of the Hilbert polynomial
$$
\operatorname{HilbPoly}_{\Osh(1)}(\mathcal{E}) := \chi(\mathcal{E} (N))
$$
are real. The number of negative roots, counted with multiplicity is 
$$g - \operatorname{i}(\mathcal{E}) - \dim \K(\mathcal{E}) \text{.}$$
Further
$$
\H^\ell(A,\mathcal{E}) = 0 \text{, }
$$
for all 
$$0 \leq \ell < \operatorname{i}(\mathcal{E})\text{,}$$ 
and 
$$
\H^{g-\ell}(A,\mathcal{E}) = 0 \text{,}
$$
for all 
$$0 \leq \ell < g - \operatorname{i}(\mathcal{E}) - \dim \K(\mathcal{E})\text{.}$$ 
}
\item[(ii)]{
Both $\mathcal{E}$ and $\operatorname{det}(\mathcal{E})$ are Weak Index Theorem sheaves and 
$$
\operatorname{j}(\mathcal{E}) = \operatorname{j}(\operatorname{det}(\mathcal{E})) \text{.}
$$
}
\item[(iii)]{
If $\mathcal{E}$ is nondegenerate, then $\mathcal{E}$ is an Index Theorem sheaf and 
$$\operatorname{j}(\mathcal{E}) = \operatorname{i}(\mathcal{E}) \text{.}$$
}
\item[(iv)]{
$\mathcal{E}$ is nondegenerate if and only if $\operatorname{det}(\mathcal{E})$ is nondegenerate.  In this case 
$$
\operatorname{j}(\mathcal{E}) = \operatorname{i}(\mathcal{E}) = \operatorname{i}(\operatorname{det}(\mathcal{E})) = \operatorname{j}(\operatorname{det}(\mathcal{E})) \text{.}
$$
}
\end{enumerate}
\end{theorem}
\begin{proof} 
Assertion (i) is a more general form of \cite[Proposition 2.1]{Grieve-cup-prod-ab-var}.  It is established similarly, and follows as an application of the Leray spectral sequence and the strong form of the Index Theorem, as established by Kempf and Ramanujam \cite[Appendix Theorem 2]{Mum:Quad:Eqns}.  
It is also helpful to recall that the index $\operatorname{i}(\mathcal{E})$ is independent of the choice of polarization (Lemma \ref{Index:Well:Defined:Lemma}).
 
Indeed, in light of Lemma \ref{Index:Well:Defined:Lemma}, the proof of assertion (i) is complete upon verification that $\dim \K(\mathcal{E})$ is the multiplicity of $0$ as a root of $\chi(\mathcal{E}(N))$. 
 
For that, let $r$ be the rank of $\mathcal{E}$.  Then, observe, from \cite[Corollary 7.9]{Muk78}, that 
$$
 \dim \K(\mathcal{E})  = \dim \K(\operatorname{det}(\mathcal{E}))  \text{.}
 $$
 Then since 
 $$
 \dim \K(\operatorname{det}(\K(\mathcal{E})) = \dim \K( \operatorname{det}(\mathcal{E})^{\otimes r} ) \text{,}
 $$
 the conclusion is that
 $$
 \dim \K(\mathcal{E}) = \dim \K(\operatorname{det}(\mathcal{E})^{\otimes r}) \text{.}
 $$
 
 Thus, the fact that $\dim \K(\mathcal{E})$ is the multiplicity of $0$, as a root of $\chi(\mathcal{E}(N))$, follows from the relation that
 $$
\chi(\mathcal{E}(N)) = \frac{\chi(\det(\mathcal{E}(N)))}{r^{g-1}} \text{,}
$$
see \cite[Proposition 6.12]{Muk78},
combined with the fact that $\dim \K( \operatorname{det}(\mathcal{E})^{\otimes r } )$ is the multiplicity of $0$, as a root of the polynomial
$$\chi(\det(\mathcal{E}(N))) = \chi(\operatorname{det}(\mathcal{E})^{\otimes r} \otimes \Osh(N) ) \text{.}$$

As explained in the discussion that proceeds \cite[Proposition 1.17]{Kuronya:Mustopa:2020},  assertions (ii), (iii) and (iv) are established in \cite[Propositions 6.2 and 6.3]{Gulbrandsen:2008}.  

For completeness, let us indicate the key points used to establish the result from \cite[Proof of Proposition 6.3]{Gulbrandsen:2008},  namely that each simple semihomogeneous vector bundle $\mathcal{E}$ over $A$ is a Weak Index Theorem sheaf.

Fix an isogeny
$$
f \colon B \rightarrow A
$$
which has the property that
$$
f_*(\mathcal{L}) \simeq \mathcal{E}
$$
for some line bundle $\mathcal{L}$ on $B$.  Let 
$$
\hat{f} \colon \hat{A} \rightarrow \hat{B}
$$
denote the dual isogeny. Let $\mathcal{Q}$ denote the normalized Poincar\'{e} line bundle on $B \times \hat{B}$.

By Proposition \ref{line:bundle:index}, $\mathcal{L}$ is a Weak Index Theorem sheaf on $B$.  On the other hand, for each $i$,
$$
\R^i{p_{\hat{A}*}}(p_A^*(f_* \mathcal{L})\otimes\mathcal{P}) \simeq \hat{f}^*(\R^i{p_{\hat{B}*}}(p_B^*(\mathcal{L})\otimes\mathcal{Q})) \text{.}
$$
Thus 
$$
\R^i{p_{\hat{A}*}}(p_A^*(\mathcal{E})\otimes\mathcal{P}) \simeq \hat{f}^*(\R^i{p_{\hat{B}*}}(p_B^*(\mathcal{L})\otimes\mathcal{Q}))
$$
and so $\mathcal{E}$ is a Weak Index Theorem sheaf and both $\mathcal{E}$ and $\mathcal{L}$ have the same weak index.

Next, we establish the equality
$$
\operatorname{j}(\mathcal{E}) = \operatorname{j}(\operatorname{det}(\mathcal{E})) \text{.}
$$
Let
$$
d := \operatorname{deg}(f)
$$
be the degree of $f$.  Then
$
f^*(\operatorname{det}(\mathcal{E}))
$ is algebraically equivalent to  $\mathcal{L}^{\otimes d}$
and thus
$$
\operatorname{j}(f^*(\operatorname{det}(\mathcal{E}))) = \operatorname{j}(\mathcal{L}) \text{,}
$$
since 
$$
\dim \K(\mathcal{L}^{\otimes d}) = \dim \K(\mathcal{L}) \text{.}
$$
(Apply \cite[Corollary 5.3]{Gulbrandsen:2008} or Proposition \ref{line:bundle:index}.)

On the other hand
$$
\R^i{p_{\hat{B}*}}(p_B^*(f^*(\operatorname{det}(\mathcal{E})))\otimes \mathcal{Q}) \simeq \hat{f}_*(\R^i{p_{\hat{A}*}}(p_A^*(\operatorname{det}(\mathcal{E})))\otimes \mathcal{P})
$$
for all $i$.  Thus
$$
\operatorname{j}(f^*(\operatorname{det}(\mathcal{E}))) = \operatorname{j}(\operatorname{det}(\mathcal{E})) \text{.}
$$
Finally, with respect to the dual isogeny, $\hat{f}$ each 
$\hat{y} \in \hat{B}$ 
is a pullback 
$$\hat{y}=f^*(\hat{x})$$ 
for some 
$\hat{x} \in \hat{A}\text{.}$  Thus, by the projection formula
$$
\H^i(B,\mathcal{L} \otimes \mathcal{Q}_{\hat{y}}) \simeq \H^i(A,f_*(\mathcal{L} \otimes \mathcal{Q}_{\hat{y}})) \simeq \H^i(A,\mathcal{E} \otimes \mathcal{P}_{\hat{x}}) \text{.}
$$
The above discussion implies that $\mathcal{L}$ is an Index Theorem sheaf if and only if $\mathcal{E}$ is an Index Theorem sheaf.  But, if $\mathcal{E}$ is nondegenerate, then so is $\operatorname{det}(\mathcal{E})$.  Thus
$$
d^g \cdot \chi(\mathcal{L}) = \chi\left(\mathcal{L}^{\otimes d}\right) = d \cdot \chi(\operatorname{det}(\mathcal{E})) \not = 0 \text{.}
$$
Thus, $\mathcal{L}$ is nondegenerate too.  Thus $\mathcal{L}$ is an Index Theorem sheaf.
It follows that $\mathcal{E}$ is an Index Theorem sheaf too.
\end{proof}

\section{Proof of Theorem \ref{R-R-index-lb}, Corollary \ref{Semi:homog:RR} and Theorem  \ref{reduced:norm:cont:reg}}\label{proof:reduced:norm:cont:reg}

In this section, we prove Theorem \ref{R-R-index-lb}, Corollary \ref{Semi:homog:RR} and Theorem \ref{reduced:norm:cont:reg}.  

\begin{proof}[Proof of Theorem \ref{R-R-index-lb}] 
Assertion (i) follows from \cite[Theorem 4.1]{Grieve:R-R:abVars} and assertion (ii) follows from \cite[Theorem 4.4]{Grieve:R-R:abVars}.  For (iii), let $D$ be the class of $\mathcal{L}$ in $\operatorname{NS}(B)$.  Then, by (ii), $\operatorname{i}(\mathcal{L})$ is the number of positive roots of the polynomial $p_{D,f^*\lambda}(n)$.  Moreover, this polynomial has a total of $g-\operatorname{i}(\mathcal{L}) - \dim \K(\mathcal{L})$ negative roots.  Thus, assertion (iii) follows from \cite[Appendix Theorem 2]{Mum:Quad:Eqns}.
\end{proof}

\begin{proof}[Proof of Corollary \ref{Semi:homog:RR}] 
 By Theorem \ref{semi:homog:index:thm},  the index of a simple semihomogeneous vector bundle $\mathcal{E}$ is equal to the index of its determinant line bundle $\operatorname{det}(\mathcal{E})$.  Thus, Corollary \ref{Semi:homog:RR} follows from Theorem \ref{R-R-index-lb}.
\end{proof}

Before establishing Theorem \ref{reduced:norm:cont:reg}, let us recall the concept of  \emph{continuous Castelnuovo-Mumford regularity} for coherent sheaves $\mathcal{F}$ on an Abelian variety $A$.  Indeed, fixing an ample line bundle $\Osh(1)$ on $A$,   the \emph{continuous Castelnuovo-Mumford regularity} of $\mathcal{F}$, with respect to $\Osh(1)$, is defined to be
$$
\operatorname{reg}_{cont}(\mathcal{F},\mathcal{O}(1)) := \min\{m \in \ZZ : \operatorname{V}^i(A,\mathcal{F}(m-i)) \not= \hat{A} \text{ for all $i > 0$}\} \text{.}
$$

\begin{proof}[Proof of {Theorem \ref{reduced:norm:cont:reg}}]   First of all, recall  that the polarization $\lambda$ allows for the identification
\begin{equation}\label{NS:End:identification}
\operatorname{NS}_{\mathbb{Q}}(B) := \operatorname{NS}(B) \otimes_{\mathbb{Z}}\mathbb{Q} \simeq \operatorname{End}^0_{f^*\lambda}(B)\text{.}
\end{equation}
Let $\eta$ be the class of $f^*\lambda$ in $\operatorname{NS}_{\mathbb{Q}}(B)$.  Then, under the isomorphism \eqref{NS:End:identification}, the class $\eta$ is identified with $\operatorname{id}_B$ the identity morphism of $B$.  

Now let  $\mathcal{E}$ be a semihomogeneous vector bundle on $B$ and assume that the class
$$
\frac{\operatorname{det}(\mathcal{E})}{\operatorname{rank}(\mathcal{E})} \in \NS_{\QQ}(B)
$$
is identified with $\gamma$.

Then, since $\mathcal{E}$ is a semihomogeneous vector bundle on $B$, it can be written in the form
$$
\mathcal{E} \simeq \bigoplus_{j} \mathcal{E}_j \otimes \mathcal{U}_j
$$
where each of the $\mathcal{E}_j$ are simple semihomogeneous vector bundles and where each of the $\mathcal{U}_j$ are unipotent.  (See \cite[Proposition 6.18]{Muk78} or \cite[Proposition 1.13]{Kuronya:Mustopa:2020}.) In particular, each of the unipotent bundles $\mathcal{U}_j$ admits a filtration of the form
$$
0 = \mathcal{U}_{0,r_j}\subseteq \hdots \subseteq \mathcal{U}_{j , r_j} = \mathcal{U}_j
$$
where each of the successive quotients
$$
\mathcal{U}_{k+1,j} / \mathcal{U}_{k,j} \text{,}
$$
for $0 \leq k \leq r_j -1$,
are trivial.  Moreover, 
inside of $\NS_\QQ(B)$, the classes  
$$\frac{\det(\mathcal{E})}{\operatorname{rank}(\mathcal{E})}$$ 
and 
$$\frac{\operatorname{det}(\mathcal{E}_j \otimes \mathcal{U}_j)}{\operatorname{rank}(\mathcal{E})}$$ 
are the  same for all $j$.  Thus, by Theorem \ref{Kuronya:Mustopa:ThmA} (i),
$$
\operatorname{reg}_{\mathrm{cont}}(\mathcal{E},\Osh(1)) = \operatorname{reg}_{\mathrm{cont}}(\mathcal{E}_j,\Osh(1)) \text{.}
$$
To finish the proof of Theorem \ref{reduced:norm:cont:reg}, it thus suffices to treat the case that  $\mathcal{E}$ is a simple semihomogeneous vector bundle.  

To that end, let $\mathcal{E}$ be a simple semihomogeneous vector bundle on $B$ with the property that the class
$$
 \frac{ \operatorname{det}(\mathcal{E})}{ \operatorname{rank}(\mathcal{E}) } \in \NS_\QQ(B)
$$
is identified with $\gamma$.

By Theorem \ref{Kuronya:Mustopa:ThmA} (i),
$$
\operatorname{reg}_{\mathrm{cont}}(\mathcal{E},f^*\lambda) = \rho_\eta \left( \gamma \right)\text{,}
$$
where $\rho_\eta(\gamma)$ is the minimum integer $m \in \mathbb{Z}$ for which either the class 
\begin{equation}\label{cont:CM:reg:class}
\gamma + (m-i)\eta
\end{equation}
is degenerate or fails to have index $i$, for all 
$$i \in \{1,\dots,g\}.$$
Let 
$$\alpha \in \operatorname{End}^0_\lambda(A)$$ 
be the image of 
$$\gamma \in \operatorname{End}^0_{f^*\lambda}(B)$$ 
under the natural homomorphism that is induced by $f$.  

Then, by Theorem \ref{R-R-index-lb}, these above conditions on the class \eqref{cont:CM:reg:class} 
translate into the assertion that either 
\begin{equation}\label{reduced:norm:degenerate}
\operatorname{pNrd}_\lambda((m-i)\id_A + \alpha) = 0
\end{equation}
or the polynomial
\begin{equation}\label{reduced:norm:index}
\operatorname{pNrd}_\lambda((N+m-i)\id_A + \alpha)
\end{equation}
fails to have $i$ positive roots (counted with multiplicities) for all 
$$i \in \{1,\dots,g\}\text{.}$$
Indeed, that the class \eqref{cont:CM:reg:class} is degenerate means that it has zero Euler characteristic
\begin{equation}\label{reduced:norm:degenerate:eqn2}
\chi(\gamma + (m-i) \eta) = 0 \text{.}
\end{equation}
By Theorem \ref{R-R-index-lb}, this is equivalent to vanishing of the polynomial \eqref{reduced:norm:degenerate}.  

Similarly, the condition that the class \eqref{cont:CM:reg:class} fails to have index $i$, for all 
$$i \in \{1,\dots,g\}\text{,}$$
means that the polynomial \eqref{reduced:norm:index} fails to have $i$ positive roots (counted with multiplicities) for all 
$$i \in \{1,\dots,g\}\text{.}$$
\end{proof}

\providecommand{\bysame}{\leavevmode\hbox to3em{\hrulefill}\thinspace}
\providecommand{\MR}{\relax\ifhmode\unskip\space\fi MR }
\providecommand{\MRhref}[2]{%
  \href{http://www.ams.org/mathscinet-getitem?mr=#1}{#2}
}
\providecommand{\href}[2]{#2}

\end{document}